\documentclass[12pt]{amsart}
\usepackage{amsmath,amsthm,amsfonts,amssymb,mathrsfs}
\date{\today}

\usepackage[cp1251]{inputenc}         

\usepackage[ukrainian]{babel} 
\usepackage{amsmath,amsthm,amsfonts,amssymb,mathrsfs}
\usepackage{eucal}

\usepackage{color}

\usepackage{amssymb,cite}
\usepackage[colorlinks,plainpages,citecolor=magenta, linkcolor=blue, backref]{hyperref}

\usepackage{hyperref}

\newtheorem{theorem}{Теорема}[section]

\newtheorem{proposition}[theorem]{Твердження}
\newtheorem{corollary}[theorem]{Наслiдок}
\newtheorem{lemma}[theorem]{Лема}
\theoremstyle{definition}
\newtheorem{remark}[theorem]{Зауваження}

\begin{document}

\title[Про гомоморфiзми бiциклiчних розширень архiмедових лiнiйно впорядкованих груп]{Про гомоморфiзми бiциклiчних розширень архiмедових лiнiйно впорядкованих груп}

\author[Олег~Гутік, Оксана Прохоренкова]{Олег~Гутік, Оксана Прохоренкова}
\address{Львівський національний університет ім. Івана Франка, Університецька 1, Львів, 79000, Україна}
\email{oleg.gutik@lnu.edu.ua, okcana.proxorenkova@gmail.com
}



\begin{abstract}
Нехай $\mathscr{B}^+(G)$~--- біциклічне розширення лінійно впорядкованої групи $G$, означене в \cite{Gutik-Pagon-Pavlyk-2011}. Ми доводимо, якщо $G$ i $H$~--- архімедові лінійно впорядковані групи, то кожний $o$-гомоморфізм $\widehat{\varphi}\colon G\to H$ породжує гомоморфізм моноїдів $\widetilde{\varphi}\colon \mathscr{B}^+(G)\to \mathscr{B}^+(H)$, і кожний гомоморфізм моноїдів $\widetilde{\varphi}\colon \mathscr{B}^+(G)\to \mathscr{B}^+(H)$ породжує $o$-гомоморфізм $\widehat{\varphi}\colon G\to H$.

\bigskip
\noindent
\emph{Oleg Gutik, Oksana Prokhorenkova, \textbf{On homomorphisms of bicyclic extensions of archimedean totally ordered groups}.}

\smallskip
\noindent
Let $\mathscr{B}^+(G)$ be the bicyclic extension of a totally ordered group $G$ which is defined in \cite{Gutik-Pagon-Pavlyk-2011}.
We show that if $G$ and $H$ are archimedean totally ordered groups then every $o$-homomorphism $\widehat{\varphi}\colon G\to H$ generates a monoid homomorphim $\widetilde{\varphi}\colon \mathscr{B}^+(G)\to \mathscr{B}^+(H)$, and every monoid homomorphism $\widetilde{\varphi}\colon \mathscr{B}^+(G)\to \mathscr{B}^+(H)$ generates $o$-homomorphism $\widehat{\varphi}\colon G\to H$.
\end{abstract}

\maketitle


\section{\textbf{Вступ, означення та мотивація досліджень}}\label{section-1}

У цій праці ми користуватимемося термінологією з \cite{Clifford-Preston-1961, Clifford-Preston-1967, Darnel-1995, Fuchs1963, Lawson-1998, Petrich1984}.
Надалі у тексті множину невід'ємних цілих чисел  позначатимемо через $\omega$, а адитивну групу цілих чисел зі звичайним на ній лінійним порядком --- через $\mathbb{R}$. Також, якщо $\le$~--- передпорядок на множині $X$, то через $\le^*$ будемо позначати дуальний передпорядок на $X$, тобто $a\le^*b$ тоді і лише тоді, коли $b\le a$, $a, b\in X$. Очевидно, що дуальний порядок до лінійного порядку знову є лінійним.

Якщо $S$~--- напівгрупа, то її підмножина ідемпотентів позначається через $E(S)$.  На\-пів\-гру\-па $S$ називається \emph{інверсною}, якщо для довільного її елемента $x$ існує єдиний елемент $x^{-1}\in S$ такий, що $xx^{-1}x=x$ та $x^{-1}xx^{-1}=x^{-1}$ \cite{Vagner-1952, Petrich1984}. В інверсній напівгрупі $S$ вище означений елемент $x^{-1}$ називається \emph{інверсним до} $x$. \emph{В'язка}~--- це напівгрупа ідемпотентів, а \emph{напівґратка}~--- це комутативна в'язка.

Якщо $S$ --- напівгрупа, то ми позначатимемо відношення Ґріна на $S$ через $\mathscr{R}$, $\mathscr{L}$, $\mathscr{D}$, $\mathscr{H}$ і $\mathscr{J}$:
\begin{align*}
    &\qquad a\mathscr{R}b \mbox{ тоді і лише тоді, коли } aS^1=bS^1;\\
    &\qquad a\mathscr{L}b \mbox{ тоді і лише тоді, коли } S^1a=S^1b;\\
    &\qquad a\mathscr{J}b \mbox{ тоді і лише тоді, коли } S^1aS^1=S^1bS^1;\\
    &\qquad \mathscr{D}=\mathscr{L}\circ\mathscr{R}=\mathscr{R}\circ\mathscr{L};\\
    &\qquad \mathscr{H}=\mathscr{L}\cap\mathscr{R}.
\end{align*}
(див. означення в \cite[\S2.1]{Clifford-Preston-1961} або \cite{Green-1951}). Напівгрупа $S$ називається \emph{простою}, якщо $S$ не містить власних двобічних ідеалів, тобто $S$ складається з одного $\mathscr{J}$-класу, і \emph{біпростою}, якщо $S$ складається з одного $\mathscr{D}$-класу. \emph{Передпорядки Ґріна} $\le_{\mathscr{L}}$ i $\le_{\mathscr{R}}$ на напівгрупі $S$ визначаються так:
\begin{align*}
  a\le_{\mathscr{L}}b & \hbox{~тоді і лише тоді, коли~} S^1a\subseteq S^1b;\\
  a\le_{\mathscr{R}}b & \hbox{~тоді і лише тоді, коли~} aS^1\subseteq bS^1,
\end{align*}
де $a,b\in S$ \cite{Grille-1995}.

\smallskip

Відношення еквівалентності $\mathfrak{K}$ на напівгрупі $S$ називається \emph{конгруенцією}, якщо для елементів $a$ та $b$ напівгрупи $S$ з того, що виконується умова $(a,b)\in\mathfrak{K}$ випливає, що $(ca,cb), (ad,bd) \in\mathfrak{K}$, для довільних $c,d\in S$. Відношення $(a,b)\in\mathfrak{K}$ ми також будемо записувати $a\mathfrak{K}b$, і в цьому випадку будемо говорити, що \emph{елементи $a$ i $b$ є $\mathfrak{K}$-еквівалентними}.

\smallskip

Якщо $S$~--- напівгрупа, то на $E(S)$ визначено частковий порядок:
$
e\preccurlyeq f
$   тоді і лише тоді, коли
$ef=fe=e$.
Так означений частковий порядок на $E(S)$ називається \emph{при\-род\-ним}.

\smallskip

Означимо відношення $\preccurlyeq$ на інверсній напівгрупі $S$ так:
$
    s\preccurlyeq t
$
тоді і лише тоді, коли $s=te$.
для деякого ідемпотента $e\in S$. Так означений частковий порядок назива\-єть\-ся \emph{при\-род\-ним част\-ковим порядком} на інверсній напівгрупі $S$~\cite{Vagner-1952}. Очевидно, що звуження природного часткового порядку $\preccurlyeq$ на інверсній напівгрупі $S$ на її в'язку $E(S)$ є при\-род\-ним частковим порядком на $E(S)$.

\smallskip

Нагадаємо (див.  \cite[\S1.12]{Clifford-Preston-1961}, що \emph{біциклічною напівгрупою} (або \emph{біциклічним моноїдом}) ${\mathscr{C}}(p,q)$ називається напівгрупа з одиницею, породжена двоелементною мно\-жи\-ною $\{p,q\}$ і визначена одним  співвідношенням $pq=1$. Біциклічна на\-пів\-група відіграє важливу роль у теорії
на\-півгруп. Так, зокрема, класична теорема О.~Ан\-дерсена \cite{Andersen-1952}  стверджує, що {($0$-)}прос\-та напівгрупа з (ненульовим) ідем\-по\-тен\-том є цілком {($0$-)}прос\-тою тоді і лише тоді, коли вона не містить ізоморфну копію бі\-циклічного моноїда. У \cite{Gutik-Prokhorenkova-Sekh-2021} описано структуру напівгрупи ендоморфізмів біциклічної напівгрупи та розширеної біциклічної напівгрупи.
Різні розширення біциклічного моноїда вводилися раніше різ\-ни\-ми авторами \cite{Fortunatov-1976, Fotedar-1974, Fotedar-1978, Warne-1967}. Такими є, зокрема, конструкції Брука та Брука--Рейлі занурення напівгруп у прості та описання інверсних біпростих і $0$-біпростих $\omega$-напівгруп \cite{Bruck-1958, Reilly-1966, Warne-1966, Gutik-2018}.

 \begin{remark}\label{remark-10}
Легко бачити, що біциклічний моноїд ${\mathscr{C}}(p,q)$ ізоморфний напівгрупі, заданій на множині $\boldsymbol{B}_{\omega}=\omega\times\omega$ з напівгруповою операцією
\begin{align*}
  (i_1,j_1)\cdot(i_2,j_2)&=(i_1+i_2-\min\{j_1,i_2\},j_1+j_2-\min\{j_1,i_2\})=\\
  &=
\left\{
  \begin{array}{ll}
    (i_1-j_1+i_2,j_2), & \hbox{якщо~} j_1\leqslant i_2;\\
    (i_1,j_1-i_2+j_2), & \hbox{якщо~} j_1\geqslant i_2.
  \end{array}
\right.
\end{align*}
\end{remark}

Нагадаємо~\cite{Fuchs1963}, що \emph{частково впорядкована група} це група $(G,\cdot)$ на якій визначено частковий порядок ``$\leqslant$'', що є інваріантним стосовно зсувів, тобто ``$\leqslant$'' задовольняє таку властивість: для всіх $a, b, g\in G$ з $a\leqslant b$ випливає, що $a\cdot g\leqslant b\cdot g$ і $g\cdot a\leqslant g\cdot b$.

\smallskip

Далі в тексті через $e$ ми будемо позначати одиничний (нейтральний) елемент групи $G$. Підмножина $G^+=\{x\in G\mid e\leqslant x\}$ частково впорядкованої групи  $G$ називається \emph{додатним конусом} в $G$ та задовольняє такі властивості:
\begin{itemize}
  \item[1)] $G^+\cdot G^+\subseteq G^+$;
  \item[2)] $G^+\cap (G^+)^{-1}=\{e\}$; \; і
  \item[3)] $x^{-1}\cdot G^+\cdot x\subseteq G^+$ для всіх $x\in G$.
\end{itemize}
Кожна підмножина $P$  групи $G$, яка задовольняє умови 1)--3) індукує частковий порядок на $G$ ($x\leqslant y$ тоді і лише тоді, коли $x^{-1}\cdot y\in P$), для якої $P$ є додатним конусом.

\smallskip

Будемо говорити, що алгебричний гомоморфізм  $h\colon G\to H$ частково впорядкованих груп (напівгруп) $G$ i $H$ є \emph{$o$-гомоморфізмом}, якщо відображення $h$ зберігає порядок (є ізотонним)~\cite{Fuchs1963}. Алгебричний ізоморфізм  $h\colon G\to H$ частково впорядкованих груп $G$ i $H$, який є порядковим ізоморфізмом, називається \emph{$o$-ізоморфізмом}. Якщо існує $o$-ізоморфізм $h\colon G\to H$ частково впорядкованих груп (напівгруп) $G$ i $H$, то у цьому випадку будемо говорити, що частково впорядковані групи $G$ i $H$ є \emph{$o$-ізоморфіними}. Також, $o$-гомоморфізм ($o$-ізоморфізм) $h\colon G\to G$ для частково впорядкованої групи $G$ будемо називати \emph{$o$-ендоморфізмом} (\emph{$o$-авторфізмом}).

\smallskip

\emph{Лінійно впорядкована} група -- це частково впорядкована група $G$ така, що відношення часткового порядку ``$\leqslant$'' на $G$ є лінійним~\cite{Birkhoff-1973}. Частково впорядкована група $G$ називається \emph{ґратково впорядкованою} або $l$-групою, якщо частковий порядок ``$\leqslant$'' визначає ґраткову структуру на $G$ \cite{Darnel-1995}. Очевидно, що кожна лінійно впорядкована група є ґратково впорядкованою. Лінійно впорядкована група $G$ називається \emph{архімедовою}, якщо для довільних $a,b\in G^+\setminus\{e\}$ існує таке натуральне число $n$, що $a\leqslant b^n$~\cite{Birkhoff-1973}. За теоремою Гьольдера (див. \cite[Theorem~24.16]{Darnel-1995} або \cite{Holder-1901}) кожна архімедова лінійно впорядкована група $o$-ізоморфна підгрупі адитивної групи дійсних чисел $\mathbb{R}$ із звичайним лінійним порядком.

\smallskip

Надалі в тексті ми будемо вважати, що $G$ лінійно впорядкована група.

\smallskip

Для кожного елемента $g\in G$ позначимо
\begin{equation*}
    G^+(g)=\{x\in G\colon g\leqslant x\}.
\end{equation*}
Множина $G^+(g)$ називається \emph{додатним конусом над елементом} $g$ в групі $G$.

\smallskip

Для довільних елементів $g,h\in G$ ми розглянемо часткове перетворення $\alpha_h^g\colon G\rightharpoonup G$, означене за формулою
\begin{equation*}
    (x)\alpha_h^g=x\cdot g^{-1}\cdot h, \qquad \hbox{ для } \; x\in
    G^{+}(g).
\end{equation*}
Зауважимо, що з леми~XIII.1 з \cite{Birkhoff-1973} випливає, що для так визначеного часткового відображення $\alpha_h^g\colon G\rightharpoonup G$ звуження
$\alpha_h^g\colon G^+(g)\rightarrow G^+(h)$ є бієктивним відображенням.

Позначимо
\begin{equation*}
    \mathscr{B}(G)=\{\alpha_h^g\colon G\rightharpoonup G\colon g,h\in
    G\} \, \hbox{ і } \,
    \mathscr{B}^+(G)=\{\alpha_h^g\colon G\rightharpoonup G\colon g,h\in
    G^+\},
\end{equation*}
і  на множинах  $\mathscr{B}(G)$ і $\mathscr{B}^+(G)$ визначимо операцію композиція часткових відображень. Легко бачити, що
\begin{equation}\label{formula-1.1}
\alpha_h^g\cdot \alpha^k_l=\alpha^a_b, \qquad \hbox{ де } \quad a=(h\vee k)\cdot h^{-1}\cdot g \quad \hbox{ і } \quad b=(h\vee
k)\cdot k^{-1}\cdot l,
\end{equation}
для $g,h,k,l\in G$. Отже, з властивісті 1) додатного конуса та умови~\eqref{formula-1.1} випливає, що  $\mathscr{B}(G)$ і $\mathscr{B}^+(G)$ є піднапівгрупами симетричного інверсного моноїда $\mathscr{I}_G$ над множиною $G$.

\smallskip

За твердженням~1.2 з \cite{Gutik-Pagon-Pavlyk-2011} для лінійно впорядкованої групи $G$ виконуються такі твердження:
\begin{itemize}
  \item[$(i)$] елементи $\alpha_h^g$ і $\alpha_g^h$ є взаємно інверсними в $\mathscr{B}(G)$ для всіх $g,h\in G$ (відп., в $\mathscr{B}^+(G)$ для всіх $g,h\in G^+$);

  \item[$(ii)$] елемент $\alpha_h^g$ напівгрупи $\mathscr{B}(G)$ (відп., $\mathscr{B}^+(G)$) є ідемпотентом тоді і лише тоді, коли $g=h$;

  \item[$(iii)$] $\mathscr{B}(G)$ і $\mathscr{B}^+(G)$ інверсні піднапівгрупи в  $\mathscr{I}_G$;

  \item[$(iv)$] напівгрупа $\mathscr{B}(G)$ (відп.,    $\mathscr{B}^+(G))$ ізоморфна напівгрупі, визначеній на множині $S_G=G\times G$ (відп.,   $S_G^+=G^+\times G^+$), з такою операцією:
  \begin{equation}\label{formula-1.2}
  (a,b)(c,d)=
  \left\{
    \begin{array}{ll}
      (c\cdot b^{-1}\cdot a,d), & \hbox{якщо }  \; b<c;\\
      (a,d), & \hbox{якщо } \; b=c; \\
      (a,b\cdot c^{-1}\cdot d), & \hbox{якщо } \; b>c,
    \end{array}
  \right.
  \end{equation}
  де $a,b,c,d\in G$ (відп.,  $a,b,c,d\in G^+$).
\end{itemize}

Очевидно, що:
\begin{itemize}
  \item[$(1)$] якщо група $G$ ізоморфна адитивній групі цілих чисел $(\mathbb{Z},+)$ зі звичайним лінійним порядком $\leqslant$, то напівгрупа $\mathscr{B}^+(G)$ ізоморфна біциклічному моноїду ${\mathscr{C}}(p,q)$, а напівгрупа $\mathscr{B}^+(G)$ ізоморфна розширеній біциклічній напівгрупі $\mathscr{C}_{\mathbb{Z}}$ (див. \cite{Fihel-Gutik-2011});

  \item[$(2)$] якщо група $G$ ізоморфна адитивній групі дійсних чисел $(\mathbb{R},+)$ зі звичайним лінійним порядком $\leqslant$, то напівгрупа
  $\mathscr{B}(G)$ ізоморфна напівгрупі  $B^{2}_{(-\infty,\infty)}$ (див. \cite{Korkmaz-1997, Korkmaz-2009}), а напівгрупа $\mathscr{B}^+(G)$ ізоморфна напівгрупі   $B^{1}_{[0,\infty)}$ (див.   \cite{Ahre-1981, Ahre-1983, Ahre-1984, Ahre-1986, Ahre-1989}), \; і

  \item[$(3)$] напівгрупа $\mathscr{B}^+(G)$ зоморфна напівгрупі $S(G)$, яка визначена в \cite{Fotedar-1974, Fotedar-1978}.
\end{itemize}

У \cite{Fotedar-1974, Gutik-Pagon-Pavlyk-2011} досліджується структура напівгруп $\mathscr{B}(G)$ і $\mathscr{B}^+(G)$ для лінійно впорядкованої групи $G$. Зокрема, описані відношення Ґріна на $\mathscr{B}(G)$ і $\mathscr{B}^+(G)$, їхні в'язки та доведено, що ці напівгрупи є біпростими. Також, у  \cite{Gutik-Pagon-Pavlyk-2011} доведено, що для комутативної лінійно впорядкованої групи $G$ усі нетривіальні конгруенції на напівгрупах $\mathscr{B}(G)$ і $\mathscr{B}^+(G)$ є груповими тоді і лише тоді, коли група $G$ архімедова та описано структуру групових конгруенцій на групах $\mathscr{B}(G)$ і $\mathscr{B}^+(G)$.

\smallskip

Надалі у цій статті ми ототожнюватимемо напівгрупу $\mathscr{B}^+(G)$ з напівгрупою $S_G^+$, відповідно, з напівгруповою операцією визначеною за формулою \eqref{formula-1.2}. Оскільки $G$~--- архімедова лінійно впорядкована група, то вона за теоремою Гьольдера (див. \cite[Theorem~24.16]{Darnel-1995} або \cite{Holder-1901}) комутативна, а отже, напівгрупова операція на $\mathscr{B}^+(G)$ виглядає так
\begin{equation}\label{formula-1.2-1}
  (a,b)(c,d)=
  \left\{
    \begin{array}{ll}
      (a\cdot b^{-1}\cdot c,d), & \hbox{якщо }  \; b<c;\\
      (a,d),                    & \hbox{якщо } \; b=c; \\
      (a,b\cdot c^{-1}\cdot d), & \hbox{якщо } \; b>c,
    \end{array}
  \right.
  \end{equation}
  де $a,b,c,d\in G^+$.

\smallskip

Ми доводимо, якщо $G$ i $H$~--- архімедові лінійно впорядковані групи, то кожний $o$-гомоморфізм $\widehat{\varphi}\colon G\to H$ породжує гомоморфізм моноїдів $\widetilde{\varphi}\colon \mathscr{B}^+(G)\to \mathscr{B}^+(H)$, і кожний гомоморфізм моноїдів $\widetilde{\varphi}\colon \mathscr{B}^+(G)\to \mathscr{B}^+(H)$ породжує $o$-гомоморфізм $\widehat{\varphi}\colon G\to H$. Звідси випливає, що напівгрупа $o$-ендоморфізмів архімедової лінійно впорядкованої групи $G$ ізоморфна напівгрупі ендоморфізмів її біциклічного розширення $\mathscr{B}^+(G)$.

\section{\textbf{Гомоморфізми біциклічних розширень}}\label{section-2}

\begin{lemma}\label{lemma-2.1}
Нехай $G$ i $H$~--- підгрупи лінійно впорядкованої адитивної групи дійсних чисел $\mathbb{R}$. Тоді кожний $o$-гомоморфізм $\varphi\colon G^+\to H^+$ додатного конуса $G^+$ групи $G$ в додатний конус $H^+$ групи $H$  визначається за формулою
\begin{equation*}
  (x)\varphi=r_\varphi\cdot x,
\end{equation*}
для деякого дійсного числа $r_\varphi$.
\end{lemma}

\begin{proof}
Зауважимо, якщо $(x_0)\varphi=0$ для деякого елемента $a\in G^+\setminus\{0\}$, то оскільки для довільного елемента $a>0$ з $G^+$ існує таке натуральне число $n$, що $na_0\geqslant a$, матимемо
\begin{equation*}
  0=n(x_0)\varphi=(nx_0)\varphi\geqslant (a)\varphi\geqslant 0,
\end{equation*}
тобто $(a)\varphi=0$. Звідси випливає, що ендоморфізм $\varphi\colon G^+\to G^+$ анулюючий, тобто $(x)\varphi=0\cdot x$ для всіх $x\in G^+$.

Далі будемо вважати, що ендоморфізм $\varphi\colon G^+\to G^+$ неанулюючий. Розглянемо в додатному конусі  $G^+$ два довільні різні числа $a_1$ й $a_2$. Не зменшуючи загальності можемо вважати, що $a_1<a_2$. Тоді $0<(a_1)\varphi<(a_2)\varphi$. Доведемо, що
\begin{equation*}
  \frac{(a_1)\varphi}{(a_2)\varphi}=\frac{a_1}{a_2},
\end{equation*}
звідки випливає рівність
\begin{equation*}
  \frac{(a_1)\varphi}{a_1}=\frac{(a_2)\varphi}{a_2},
\end{equation*}
яку нам і треба було довести. Припустимо, що наприклад
\begin{equation*}
  \frac{(a_1)\varphi}{(a_2)\varphi}<\frac{a_1}{a_2}.
\end{equation*}
Тоді за принципом Архімеда існує додатне таке раціональне число $\dfrac{m}{n}$, що
\begin{equation*}
  \frac{(a_1)\varphi}{(a_2)\varphi}<\dfrac{m}{n}<\frac{a_1}{a_2},
\end{equation*}
де $n,m\in\mathbb{\mathbb{N}}$. Звідси впливає, що $na_1>ma_2$ i $n(a_1)\varphi<m(a_2)\varphi$. Отримали протиріччя, оскільки з нерівності $na_1>ma_2$ випливає, що
\begin{equation*}
  n(a_1)\varphi=(n a_1)\varphi>(m a_2)\varphi=m(a_2)\varphi.
\end{equation*}
Припустивши, що
\begin{equation*}
  \frac{(a_1)\varphi}{(a_2)\varphi}>\frac{a_1}{a_2},
\end{equation*}
аналогічно попередньо викладеним міркуванням, отримуємо протиріччя.
З вище доведеного випливає, що існує таке дійсне число $r_\varphi$, що для всіх $x\in G^+$ виконується рівність $(x)\varphi=r_\varphi\cdot x$. Зауважимо, що $r_\varphi=\dfrac{(a_1)\varphi}{a_1}$.
\end{proof}



\begin{lemma}\label{lemma-2.3}
Нехай $G$ i $H$~--- лінійно впорядковані групи, причому $G$~--- архімедова. Тоді для кожного $o$-гомоморфізму $\varphi\colon G\to H$  виконується лише одна з умов:
\begin{enumerate}
  \item[$(i)$]  $\varphi$~--- ін'єктивне відображення;
  \item[$(ii)$] $\varphi$~--- анулюючий гомоморфізм.
\end{enumerate}
\end{lemma}

\begin{proof}
У випадку, коли $G$~--- тривіальна (одноелементна) група, то умови $(i)$ та $(ii)$ збігаються та твердження леми очевидне. Тому надалі будемо вважати, що  $G$~--- нескінченна група.

Припустимо, що $\varphi\colon G\to H$~--- неін'єктивний гомоморфізм. Тоді існують такі $x,y\in G$, що $x<y$ i $(x)\varphi=(y)\varphi$. Тоді $e=x\cdot x^{-1}<y\cdot x^{-1}$ i
\begin{equation*}
(e)\varphi=(x\cdot x^{-1})\varphi=(x)\varphi\cdot (x^{-1})\varphi=(y)\varphi\cdot (x^{-1})\varphi=(y\cdot x^{-1})\varphi,
\end{equation*}
а отже, з архімедовості лінійного порядку на групі $G$ випливає, що $(e)\varphi=(g)\varphi$ для довільного $g\in G^+$. Оскільки $G$~--- лінійно впорядкована група, то для довільного елемента $g_1\in \in G^+$ існує натуральне число $n$ таке, що $g_1\leqslant g^n$. Звідси випливає, що
\begin{equation*}
(e)\varphi\leqslant(g_1)\varphi\leqslant(g^n)\varphi=((g)\varphi)^n=((e)\varphi)^n=(e)\varphi,
\end{equation*}
а отже, $\varphi$~--- анулюючий гомоморфізм.
\end{proof}

Для довільної лінійно впорядкованої групи $G$ означимо
\begin{equation*}
  \mathscr{B}^+_{\rightarrow}(G)=\left\{(e,g)\colon g\in G^+\right\} \qquad \hbox{i} \qquad \mathscr{B}^+_{\downarrow}(G)=\left\{(g, e)\colon g\in G^+\right\}.
\end{equation*}
Для довільного елемента $g\in G^+$ маємо, що
\begin{align*}
  \mathscr{B}^+(G)\cdot(e,g)& =\left\{(x,y)\cdot(e,g)\colon x,y\in G^+\right\}= \\
   &=\left\{(x,y\cdot g)\colon x,y\in G^+\right\}=\\
   &=G^+\times (G^+\cdot g)
\end{align*}
i
\begin{align*}
  (g,e)\cdot\mathscr{B}^+(G)&=\left\{(g,e)\cdot(x,y)\colon x,y\in G^+\right\}=  \\
  &=\left\{(g\cdot x,y)\colon x,y\in G^+\right\}=  \\
  &=(g\cdot G^+)\times G^+.
\end{align*}
Оскільки за теоремою Вагнера--Престона (див. \cite[Theorem~1.17]{Clifford-Preston-1961}) кожен головний лівий  ідеал і кожен головний правий  ідеал в інверсній напівгрупі породжується єдиним ідемпотентом, то з лінійності природного часткового порядку на напівґратці ідемпотентів напівгрупи $\mathscr{B}^+(G)$ (див. \cite[Proposition~2.1$(i)$]{Gutik-Pagon-Pavlyk-2011}) випливає таке тверд\-ження.

\begin{proposition}\label{proposition-2.4}
Нехай $G$ ~--- лінійно впорядкована група. Для $g_1,g_2\in G^+$ такі умови еквівалентні:
\begin{enumerate}
  \item[$(i)$] $(e,g_1)\le_{\mathscr{L}}^*(e,g_2)$;
  \item[$(ii)$] $(g_1,e)\le_{\mathscr{R}}^*(g_2,e)$;
  \item[$(iii)$] $(g_2,g_2)\preccurlyeq(g_1,g_1)$;
  \item[$(iv)$] $g_1\leqslant g_2$,
\end{enumerate}
а отже, $\left(\mathscr{B}^+_{\rightarrow}(G),\le_{\mathscr{L}}^*\right)$ i $\left(\mathscr{B}^+_{\downarrow}(G),\le_{\mathscr{R}}^*\right)$~--- лінійно впорядковані множини.
\end{proposition}

\begin{lemma}\label{lemma-2.5}
Нехай $G$ ~--- лінійно впорядкована група. Тоді підмножина $\mathscr{B}^+_{\rightarrow}(G)$ $(\mathscr{B}^+_{\downarrow}(G))$
з індукованою напівгруповою операцією з $\mathscr{B}^+(G)$ і лінійним порядком $\le_{\mathscr{L}}^*$ $(\le_{\mathscr{R}}^*)$ $o$-ізоморфна додатному конусу $G^+$ з індукованою напівгруповою операцією з групи $G$.
\end{lemma}

\begin{proof}
Очевидно, що відображення $\iota_{\rightarrow}\colon G^+\to \mathscr{B}^+_{\rightarrow}(G)$ означене за формулою $(g)\iota_{\rightarrow}=(e,g)$ є $o$-ізоморфізмом. Справді, легко бачити, що $\iota_{\rightarrow}$ --- бієктивне відображення, і для довільних $g_1,g_2\in G^+$ маємо, що
\begin{equation*}
\begin{split}
  (g_1)\iota_{\rightarrow}\cdot(g_2)\iota_{\rightarrow}&=(e,g_1)\cdot(e,g_2)= \\
    & =(e,g_1\cdot e^{-1}\cdot g_2)=\\
    & =(e,g_1\cdot g_2)=\\
    & =(g_1\cdot g_2)\iota_{\rightarrow},
\end{split}
\end{equation*}
оскільки $e\leqslant g_1$ у $G^+$. Далі скористаємося твердженням~\ref{proposition-2.4}.

Аналогічно доводиться, що відображення $\iota_{\downarrow}\colon G^+\to \mathscr{B}^+_{\downarrow}(G)$ означене за формулою $(g)\iota_{\downarrow}=(g,e)$ є $o$-ізоморфізмом.
\end{proof}

\begin{lemma}\label{lemma-2.6}
Нехай $G$ i $H$~--- архімедові лінійно впорядковані групи. Тоді кожний $o$-гомоморфізм $\varphi\colon G\to H$ породжує гомоморфізм моноїдів $\widetilde{\varphi}\colon \mathscr{B}^+(G)\to \mathscr{B}^+(H)$, який визначається за формулою
\begin{equation*}
 (x,y)\widetilde{\varphi}=\big((x)\varphi,(y)\varphi\big), \qquad \hbox{для всіх} \quad x,y\in G^+.
\end{equation*}
\end{lemma}

\begin{proof}
У випадку, якщо $\varphi\colon G\to H$ --- анулюючий гомоморфізм, то твердження леми оче\-вид\-не.

За\-фік\-суємо довільні $x_1,x_2,y_1,y_2\in G^+$. Тоді
\begin{align*}
  \big((x_1,y_1)\cdot(x_2,y_2)\big)\widetilde{\varphi}&=
    \left\{
      \begin{array}{ll}
        \big(x_1\cdot y_1^{-1}\cdot x_2,y_2\big)\widetilde{\varphi}, & \hbox{якщо~} y_1<x_2;\\
        \big(x_1,y_2\big)\widetilde{\varphi},                        & \hbox{якщо~} y_1=x_2;\\
        \big(x_1,y_1\cdot x_2^{-1}\cdot y_2\big)\widetilde{\varphi}, & \hbox{якщо~} y_1>x_2
      \end{array}
    \right.
    \\
   &=
    \left\{
      \begin{array}{ll}
        \big((x_1\cdot y_1^{-1}\cdot x_2)\varphi,(y_2)\varphi\big), & \hbox{якщо~} y_1<x_2;\\
        \big((x_1)\varphi,(y_2)\varphi\big),                        & \hbox{якщо~} y_1=x_2;\\
        \big((x_1)\varphi,(y_1\cdot x_2^{-1}\cdot y_2)\varphi\big), & \hbox{якщо~} y_1>x_2
      \end{array}
    \right.
\end{align*}
i, оскільки $\varphi\colon G\to H$ --- $o$-гомоморфізм лінійно впорядкованих груп, то за лемою~\ref{lemma-2.3} маємо, що
\begin{align*}
  (x_1,y_1)\widetilde{\varphi}\cdot(x_2,y_2)\widetilde{\varphi}&=\big((x_1)\varphi,(y_2)\varphi\big)\cdot\big((x_1)\varphi,(y_2)\varphi\big)= \\
   &=
    \left\{
      \begin{array}{ll}
        \big((x_1)\varphi\cdot((y_1)\varphi)^{-1}\cdot(x_2)\varphi,(y_2)\varphi\big), & \hbox{якщо~} (y_1)\varphi<(x_2)\varphi;\\
        \big((x_1)\varphi,(y_2)\varphi\big),                                          & \hbox{якщо~} (y_1)\varphi=(x_2)\varphi;\\
        \big((x_1)\varphi,(y_1)\varphi\cdot((x_2)\varphi)^{-1}\cdot(y_2)\varphi\big), & \hbox{якщо~} (y_1)\varphi>(x_2)\varphi
      \end{array}
    \right.\\
   &=
    \left\{
      \begin{array}{ll}
        \big((x_1)\varphi\cdot((y_1)\varphi)^{-1}\cdot(x_2)\varphi,(y_2)\varphi\big), & \hbox{якщо~} y_1<x_2;\\
        \big((x_1)\varphi,(y_2)\varphi\big),                                          & \hbox{якщо~} y_1=x_2;\\
        \big((x_1)\varphi,(y_1)\varphi\cdot((x_2)\varphi)^{-1}\cdot(y_2)\varphi\big), & \hbox{якщо~} y_1>x_2
      \end{array}
    \right. \\
   &=
    \left\{
      \begin{array}{ll}
        \big((x_1)\varphi\cdot(y_1^{-1})\varphi\cdot(x_2)\varphi,(y_2)\varphi\big), & \hbox{якщо~} y_1<x_2;\\
        \big((x_1)\varphi,(y_2)\varphi\big),                                        & \hbox{якщо~} y_1=x_2;\\
        \big((x_1)\varphi,(y_1)\varphi\cdot(x_2^{-1}\varphi)\cdot(y_2)\varphi\big), & \hbox{якщо~} y_1>x_2
      \end{array}
    \right.\\
   &=
    \left\{
      \begin{array}{ll}
        \big((x_1\cdot y_1^{-1}\cdot x_2)\varphi,(y_2)\varphi\big), & \hbox{якщо~} y_1<x_2;\\
        \big((x_1)\varphi,(y_2)\varphi\big),                        & \hbox{якщо~} y_1=x_2;\\
        \big((x_1)\varphi,(y_1\cdot x_2^{-1}\cdot y_2)\varphi\big), & \hbox{якщо~} y_1>x_2.
      \end{array}
    \right.
\end{align*}
Отже, так означене $\widetilde{\varphi}\colon \mathscr{B}^+(G)\to \mathscr{B}^+(H)$ є $o$-гомоморфізмом моноїдів.
\end{proof}

\begin{lemma}\label{lemma-2.7}
Нехай $G$ i $H$~--- архімедові лінійно впорядковані групи. Тоді кожний гомоморфізм моноїдів $\widetilde{\varphi}\colon \mathscr{B}^+(G)\to \mathscr{B}^+(H)$ породжує $o$-гомоморфізм $\varphi\colon G^+\to H^+$ такий, що
\begin{equation}\label{eq-2.1}
 (x,y)\widetilde{\varphi}=\big((x)\varphi,(y)\varphi\big), \qquad \hbox{для всіх} \quad (x,y)\in \mathscr{B}^+(G).
\end{equation}
\end{lemma}

\begin{proof}
За теоремою Гьольдера (див. \cite[Theorem~24.16]{Darnel-1995} або \cite{Holder-1901}), не зменшуючи загальності, можемо вважати, що $G$ i $H$~--- підгрупи адитивної групи дійсних чисел $\mathbb{R}$ із звичайним лінійним порядком.

Якщо гомоморфізм моноїдів $\widetilde{\varphi}\colon \mathscr{B}^+(G)\to \mathscr{B}^+(H)$ анулюючий, то $(x,y)\widetilde{\varphi}=(0,0)$ для всіх $(x,y)\in \mathscr{B}^+(G)$. Отже, $(g)\varphi=0$ для всіх $g\in G^+$, звідки випливає, що $(g)\varphi=0$ для довільного елемента $g\in G$.

Припустимо, що гомоморфізм моноїдів $\widetilde{\varphi}\colon \mathscr{B}^+(G)\to \mathscr{B}^+(H)$ неанулюючий та неін'єктивний. За теоремою 3.2 з \cite{Gutik-Pagon-Pavlyk-2011} кожна неодинична конгруенція на напівгрупі $\mathscr{B}^+(G)$ є груповою. Отже, образ $(\mathscr{B}^+(G))\widetilde{\varphi}$ є пігрупою напівгрупи $\mathscr{B}^+(H)$. Оскільки за твердженням 2.1$(v)$~\cite{Gutik-Pagon-Pavlyk-2011} усі $\mathscr{H}$-класи в напівгрупі $\mathscr{B}^+(H)$ одноелементні, то  $(\mathscr{B}^+(G))\widetilde{\varphi}$ --- одноелементна підмножина в $\mathscr{B}^+(H)$, тобто $\widetilde{\varphi}\colon \mathscr{B}^+(G)\to \mathscr{B}^+(H)$ --- анулюючий гомоморфізм, протиріччя. З отриманого протиріччя випливає, що
$\widetilde{\varphi}\colon \mathscr{B}^+(G)\to \mathscr{B}^+(H)$ --- ін'єктивний гомоморфізм.

Означимо $o$-гомоморфізм $\varphi\colon G^+\to H^+$.

Оскільки $\widetilde{\varphi}\colon \mathscr{B}^+(G)\to \mathscr{B}^+(H)$~--- гомоморфізм моноїдів, то
\begin{equation*}
(0,0)\widetilde{\varphi}=(0,0)=((0)\varphi,(0)\varphi),
\end{equation*}
а отже, $(0)\varphi=0$.

Для довільного $g\in G^+$ приймемо $(0,g)\widetilde{\varphi}=(0,(g)\varphi)$. Далі доведемо, що так визначене відображення $\varphi\colon G^+\to H^+$ є $o$-гомоморфізмом, який задовольняє умову \eqref{eq-2.1}. Оскільки $(0,g)\widetilde{\varphi}\in \mathscr{B}^+(H)$, то за означенням напівгрупи $\mathscr{B}^+(H)$ маємо, що $(g)\varphi\in H^+$ довільного $g\in G^+$. Тоді для довільних $g_1,g_2\in G^+$ з рівності
\begin{equation*}
  (0,g_1)\cdot(0,g_1)=(0,g_1+g_2),
\end{equation*}
випливає, що
\begin{align*}
  ((0,g_1)\cdot(0,g_1))\widetilde{\varphi}& =(0,g_1+g_2)\widetilde{\varphi}= \\
   &=((0)\varphi,(g_1+g_2)\varphi)=\\
   &=(0,(g_1+g_2)\varphi)
\end{align*}
i
\begin{align*}
  ((0,g_1)\cdot(0,g_1))\widetilde{\varphi}&= (0,g_1)\widetilde{\varphi}\cdot(0,g_1)\widetilde{\varphi}=\\
   & =((0)\varphi,(g_1)\varphi)\cdot((0)\varphi,(g_2)\varphi)=\\
   & =(0,(g_1)\varphi)\cdot(0,(g_2)\varphi)=\\
   & =(0,(g_1)\varphi+(g_2)\varphi).
\end{align*}
Отже, так означене відображення $\varphi\colon G^+\to H^+$ є гомоморфізмом. З леми~\ref{lemma-2.5} випливає, що ${\varphi}$~--- $o$-го\-мо\-морфізм з додатного конуса $G^+$ в додатний конус $H^+$.

Оскільки $\widetilde{\varphi}\colon \mathscr{B}^+(G)\to \mathscr{B}^+(H)$ --- гомоморфізм інверсних напівгруп, то з тверд\-ження~1.4.21 з \cite{Lawson-1998} випливає, що
\begin{align*}
  (g,0)\widetilde{\varphi}&=((0,g)^{-1})\widetilde{\varphi}= \\
   & =((0,g)\widetilde{\varphi})^{-1}= \\
   & =(0,(g)\varphi)^{-1}=\\
   & =((g)\varphi,0)
\end{align*}
довільного елемента $g\in G^+$. Отже,
\begin{align*}
  (x,y)\widetilde{\varphi}&=((x,0)\cdot(0,y)\widetilde{\varphi}= \\
   & =(x,0)\widetilde{\varphi}\cdot(0,y)\widetilde{\varphi}= \\
   & =((x)\varphi,0)\cdot(0,(x)\varphi)= \\
   & =((x)\varphi,(x)\varphi)
\end{align*}
для довільних $x,y\in G^+$, звідки випливає, що формула \eqref{eq-2.1} коректно визначає $o$-го\-мо\-морфізм $\varphi\colon G^+\to H^+$.
\end{proof}

Підсумуємо отримані результати в такій теоремі.

\begin{theorem}\label{theorem-2.8}
Нехай $G$ i $H$~--- архімедові лінійно впорядковані групи. Кожний $o$-го\-моморфізм $\widehat{\varphi}\colon G\to H$ породжує гомоморфізм моноїдів $\widetilde{\varphi}\colon \mathscr{B}^+(G)\to \mathscr{B}^+(H)$, і кожний гомоморфізм моноїдів $\widetilde{\varphi}\colon \mathscr{B}^+(G)\to \mathscr{B}^+(H)$ породжує $o$-гомоморфізм $\widehat{\varphi}\colon G\to H$, які узгоджуються за формулою
\begin{equation}\label{eq-2.2}
 (x,y)\widetilde{\varphi}=\big((x)\widehat{\varphi},(y)\widehat{\varphi}\big), \qquad x,y\in G^+.
\end{equation}
\end{theorem}

\begin{proof}
Перше твердження теореми виливає з леми~\ref{lemma-2.6}.

Доведемо друге твердження. За лемою~\ref{lemma-2.7} кожний гомоморфізм моноїдів $\widetilde{\varphi}\colon \mathscr{B}^+(G)\to \mathscr{B}^+(H)$ породжує $o$-гомоморфізм $\varphi\colon G^+\to H^+$, який задовольняє умову \eqref{eq-2.1}.
За твердженням 5.6 з \cite{Darnel-1995} $o$-гомоморфізм $\varphi$ продовжується до єдиного $o$-гомоморфізму лінійно впорядкованих груп $\widehat{\varphi}\colon G\to H$ такого, що $(g)\widehat{\varphi}=(g)\varphi$ для всіх $g\in G^+$.
\end{proof}

З теореми \ref{theorem-2.8} випливає такий наслідок.

\begin{corollary}
Архімедові лінійно впорядковані групи $G$ i $H$ $o$-ізоморфні тоді і лише тоді, коли моноїди $\mathscr{B}^+(G)$ i $\mathscr{B}^+(H)$ ізоморфні. Більше того, кожний $o$-ізоморфізм $\widehat{\varphi}\colon G\to H$ породжує ізоморфізм моноїдів $\widetilde{\varphi}\colon \mathscr{B}^+(G)\to \mathscr{B}^+(H)$, і кожний ізоморфізм моноїдів $\widetilde{\varphi}\colon \mathscr{B}^+(G)\to \mathscr{B}^+(H)$ породжує $o$-ізоморфізм $\widehat{\varphi}\colon G\to H$, які узгоджуються за формулою \eqref{eq-2.2}.
\end{corollary}

З теореми~\ref{theorem-2.8} випливає, що відображення $\Phi_{\mathscr{B}}$ з напівгрупи $\mathbf{End}^o(G)$ $o$-ендо\-морфізмів архімедової лінійно впорядкованої групи $G$ у напівгрупу $\mathbf{End}(\mathscr{B}^+(G))$ ендоморфізмів її біциклічного розширення $\mathscr{B}^+(G)$, визначене за формулою $(\varphi)\Phi_{\mathscr{B}}=\widetilde{\varphi}$, є ізоморфізмом. Отже, виконується така теорема

\begin{theorem}\label{theorem-2.9}
Нехай $G$ архімедова лінійно впорядкована група. Тоді напівгрупи $\mathbf{End}^o(G)$ та  $\mathbf{End}(\mathscr{B}^+(G))$ ізоморфні.
\end{theorem}

\begin{corollary}
Нехай $G$ архімедова лінійно впорядкована група. Тоді група $o$-автомофізмів групи $G$ ізоморфна групі автоморфізмів моноїда $\mathscr{B}^+(G)$.
\end{corollary}
\section{\textbf{Про категорії $\mathfrak{TOAG}$ i $\mathfrak{BETOAG}$}}\label{section-3}

Означимо категорію $\mathfrak{TOAG}$ так:
\begin{enumerate}
  \item $\mathbf{Ob}(\mathfrak{TOAG})=\left\{G\colon G~- \hbox{архімедова лінійно впорядкована група}\right\}$;
  \item $\mathbf{Mor}(\mathfrak{TOAG})$ --- $o$-гомоморфізми архімедових лінійно впорядкованих груп,
\end{enumerate}
а категорію $\mathfrak{BETOAG}$ так:
\begin{enumerate}
  \item $\mathbf{Ob}(\mathfrak{BETOAG})$ біциклічні розширення $\mathscr{B}^+(G)$ архімедових лінійно впорядкованих груп $G\in \mathbf{Ob}(\mathfrak{TOAG})$;
  \item $\mathbf{Mor}(\mathfrak{BETOAG})$ є гомоморфізми моноїдів $\mathscr{B}^+(G)$.
\end{enumerate}

Для кожного об'єкту $G\in \mathbf{Ob}(\mathfrak{TOAG})$ означимо $\mathbf{B}(G)=\mathscr{B}^+(G)$ --- біциклічне розширення ар\-хі\-ме\-до\-вої лінійно впорядкованої групи $G$. Для кожного морфізму $\varphi\colon G\to H$ з $\mathbf{Mor}(\mathfrak{TOAG})$ означимо $\mathbf{B}(\varphi)=\widetilde{\varphi}$, де гомоморфізм моноїдів $\widetilde{\varphi}\colon \mathscr{B}^+(G)\to \mathscr{B}^+(H)$, що визначається за формулою
\begin{equation}\label{eq-3.1}
 (x,y)\widetilde{\varphi}=\big((x)\varphi,(y)\varphi\big), \qquad \hbox{для всіх} \quad x,y\in G^+.
\end{equation}

З леми~\ref{lemma-2.6} випливає, що $\mathbf{B}$ є функтором з категорії $\mathfrak{TOAG}$ у категорію $\mathfrak{BETOAG}$.

Для кожного об'єкту $\mathscr{B}^+(G)\in \mathbf{Ob}(\mathfrak{BETOAG})$ означимо $\mathbf{T}(\mathbf{Е}(\mathscr{B}^+(G)))=G$ --- така ар\-хі\-ме\-до\-ва лінійно впорядкована група, що $\mathscr{B}^+(G)$ є біциклічним розширенням лінійно впорядкованої групи $G$. Для кожного морфізму $\widetilde{\varphi}\colon \mathscr{B}^+(G)\to \mathscr{B}^+(H)$ з $\mathbf{Mor}(\mathfrak{BETOAG})$ означимо $\mathbf{T}(\widetilde{\varphi})={\varphi}$, де $\varphi\colon G\to H$ $o$-гомоморфізми архімедових лінійно впорядкованих груп $G$ i $H$, який за теоремою~\ref{theorem-2.8} визначається за формулою \eqref{eq-3.1}.

Функтор $\mathbf{I}$ з категорії $\mathfrak{C}$ в категорію $\mathfrak{K}$ називається \emph{ізоморфізмом} категорій $\mathfrak{C}$ і $\mathfrak{K}$, якщо існує функтор $\mathbf{J}\colon \mathfrak{K}\to \mathfrak{C}$, для якого обидві композиції $\mathbf{I}\circ\mathbf{J}\colon \mathfrak{C}\to \mathfrak{C}$ i $\mathbf{J}\circ\mathbf{I}\colon \mathfrak{K}\to \mathfrak{K}$ є тотожними функ\-то\-ра\-ми \cite{Mac_Lane-2010}.

Очевидно, що функтори $\mathbf{B}\colon \mathfrak{TOAG}\to\mathfrak{BETOAG}$ i $\mathbf{I}\colon\mathfrak{BETOAG} \to\mathfrak{TOAG}$ є взаємно оберненими, а отже справджується така теорема:

\begin{theorem}
Категорії $\mathfrak{TOAG}$ i $\mathfrak{BETOAG}$ ізоморфні.
\end{theorem}





\end{document}